\documentclass[10pt]{amsart}
\usepackage{amsmath,amsthm,amssymb,amsfonts, eucal, amscd, mathbbol, mathrsfs, mathabx}
\usepackage[shortlabels]{enumitem}
\usepackage{cite}
\usepackage{setspace}
\usepackage[all,cmtip]{xy}{
\usepackage{graphicx}
\usepackage{subfig}
\usepackage{fancyhdr}
\usepackage{latexsym}
\usepackage{fncylab}
\usepackage{epic}
\usepackage{ifthen}
\usepackage[bookmarks,bookmarksnumbered, plainpages=false, pdfpagelabels]{hyperref}
\usepackage{xcolor} 
\usepackage{textgreek}
\hypersetup{
 colorlinks   = true, 
 urlcolor     = green, 
 linkcolor    = blue, 
 citecolor   = red 
}

\usepackage{rotating}
\usepackage{scalefnt}
\usepackage{enumitem}
\setlist{nolistsep}

\parskip = 0.2in
\parindent = 0.0in
\topmargin = 0.0in

\oddsidemargin = 0.0in
\evensidemargin = 0.0in
\textwidth = 6.0in

\newtheorem{thm}{Theorem}

\newtheorem*{thm*}{Main Theorem}
\newtheorem*{thm**}{Theorem}

\newtheorem{lem}[thm]{Lemma}

\theoremstyle{definition}

\theoremstyle{definition}

\theoremstyle{definition}

\theoremstyle{definition}

\theoremstyle{definition}

\theoremstyle{definition}

\theoremstyle{definition}

\newcommand{\R}{\ensuremath{\mathbb{R}}}
\newcommand{\N}{\ensuremath{\mathbb{N}}} 

\newcommand*{\medcup}{\mathbin{\scalebox{1}{\ensuremath{\bigcup\,}}}}%

\def\i{\infty}

\def\co{\textrm{Conv}}
\def\fr{\partial}

\def\cal{\mathcal}
\def\a{\alpha}

\def\G{\Gamma}
\def\d{\delta}
 
\def\e{\epsilon}

\def\H{\mathcal{H}}

\def\downto{\searrow}

\renewcommand{\theenumi}{(\alph{enumi})}
\makeatletter
\renewcommand{\p@enumii}{\theenumi}
\makeatother    

\makeatletter
\newcommand{\eqnum}{\refstepcounter{equation}}
\makeatother

\begin{document} 
	
\author[H. Pugh]{H. Pugh\\ Mathematics Department \\ Stony Brook University} 
\title{Reifenberg's Isoperimetric Inequality Revisited}
\begin{abstract}
	We prove a generalization of Reifenberg's isoperimetric inequality. The main result of this paper is used in \cite{elliptic} to establish existence of a minimizer for an anisotropically-weighted area functional among a collection of surfaces which satisfies a set of axioms, namely being closed under certain deformations and Hausdorff limits. This problem is known as the \emph{axiomatic Plateau problem.}
\end{abstract}

\maketitle

\section{Introduction}
Let \( n\geq m\geq 2 \) be integers and let \( \H^d \) denote \( d \)-dimensional Hausdorff measure. In \cite{reifenberg}, the following isoperimetric inequality is proved:

\begin{thm}[\cite{reifenberg} Lemma 8, Chapter 1]
	\label{thm:lem8}
	There exists a constant \( K<\i \) which depends only on \( n \), such that if \( A\subset \R^n \) is compact, there exists a compact set \( X\subset \R^n \) which spans \( A \) and satisfies \[ \H^m(X)^{m-1} \leq K\cdot \H^{m-1}(A)^m. \] Furthermore, \( X \) is contained in the convex hull of \( A \) and in the closed neighborhood of \( A \) of radius\footnote{If \( \H^{m-1}(A)=0 \), we understand the neighborhood of radius zero of \( A \) to mean the set \( A \) itself.} \( K\cdot \H^{m-1}(A)^{1/(m-1)} \).
\end{thm}

The precise definition of ``span'' used by Reifenberg is not needed here, but roughly this means that \( X \) contains \( A \) and fills in the \( m \)-dimensional holes of \( A \). We prove the following generalization of this result:

\begin{thm}
	\label{thm:main0}
	There exists a constant \( K<\i \) which depends only on \( n \) such that if \( A\subset \R^n \) and \( Y\subset \R^n \) are compact, and if \( L>0 \) is an upper bound for \( \H^{m-1}(A)^{1/(m-1)} \), then there exists a sequence of diffeomorphisms \( (\phi_i)_{i\in \N} \) of \( \R^n \) leaving \( A \) fixed such that the Hausdorff limit \( \tilde{Y} \) of \( (\phi_i(Y))_{i\in \N} \) exists and satisfies \[ \H^m(\tilde{Y})^{m-1} \leq K\cdot \H^{m-1}(A)^m. \] Furthermore, \( \tilde{Y} \) is contained, up to a closed \( (m-1) \)-rectifiable exceptional set of finite \( \H^{m-1} \) measure, in the convex hull of \( A \) and in the closed neighborhood of \( A \) of radius \( K\cdot L \). Finally, if \( A \) is \( (m-1) \)-rectifiable, then \( \tilde{Y} \) is \( m \)-rectifiable.
\end{thm}

In fact, we will prove more, namely that the diffeomorphisms \( \phi_i \) and the exceptional set are essentially independent of \( Y \). Specifically:

\begin{thm}
	\label{thm:main}
	There exists a constant \( K<\i \) which depends only on \( n \) such that if \( A\subset \R^n \) is compact and contained in a ball \( B\subset \R^n \), and if \( L>0 \) is an upper bound for \( \H^{m-1}(A)^{1/(m-1)} \), then there exist a sequence \( (\phi_i)_{i\in \N} \) of diffeomorphisms of \( \R^n \) leaving \( A\cup (\R^n\setminus B) \) fixed, and a closed \( (m-1) \)-rectifiable set \( S \) of finite \( \H^{m-1} \) measure, so that if \( Y\subset B \) is closed and \( Y\cap \fr B\subset A \), then there exists a subsequence of \( (\phi_i(Y))_{i\in \N} \) whose Hausdorff limit \( \tilde{Y} \) exists and satisfies \[ \H^m(\tilde{Y})^{m-1} \leq K\cdot \H^{m-1}(A)^m. \] Furthermore, \( \tilde{Y}\setminus S \) is contained in the convex hull of \( A \) and in the closed neighborhood of \( A \) of radius \( K\cdot L \). Moreover, \( \tilde{Y}\cap \fr B\subset A \). Finally, if \( A \) is \( (m-1) \)-rectifiable, then \( \tilde{Y} \) is \( m \)-rectifiable.
\end{thm}

We will prove even slightly more, and for details we refer the reader to the proposition \( P(m,0,n,n) \) below, which we will prove in the course of proving Theorem \ref{thm:main}.

Theorem \ref{thm:main0} generalizes Theorem \ref{thm:lem8} in the case\footnote{When \( \H^{m-1}(A)=0 \), Theorem \ref{thm:lem8} follows from dimension theory (see \cite{hurewicz} VII 3, VIII 3',) whereby such a set \( A \) has no \( m \)-dimensional holes, so the set \( X=A \) satisfies the requirements of Theorem \ref{thm:lem8}. In this case, Theorem \ref{thm:main0} gives a different set \( \tilde{Y} \) for each choice of \( L>0 \).} that \( \H^{m-1}(A)>0 \) by letting \( L \) be the constant \( \H^{m-1}(A)^{1/(m-1)} \) and \( Y \) be the inward cone\footnote{See definition below.} over \( A \) about any point \( p\in \R^n \). The exceptional set can then be removed from \( \tilde{Y} \) by a standard dimension theory argument (see \cite{reifenberg} Lemma 2A, \cite{cohomology} Lemma 1.2.18, \cite{hurewicz} VII 3) to produce the desired set \( X \).

A related, logically independent inequality is proved in \cite{almgrenannals} (2.9 (b2)) in the case that \( A=Y\cap \fr B \), where the term \( \H^{m-1}(A) \) is replaced by \( \frac{d}{dt} \H^m(X\cap B_t(p))\lfloor_{t=r} \), where \( B_t(p) \) denotes the ball of radius \( t \) about \( p \), and \( B_r(p)=B \).

\section{Notation}
If \( X\subset \R^n \) and \( p\in \R^n \), let \( C(X,p) \) denote the inward cone over \( X \) about \( p \). That is, \( C(X,p) = \bigcup_{x\in X} l(x,p) \), where \( l(x,p) \) denotes the closed line segment between \( x \) and \( p \). Let \( \co(X) \) denote the convex hull of \( X \). If \( r>0 \), let \( N(X,r) \) denote the closed neighborhood about \( X \) of radius \( r \). Let \( \fr(X) \) denote the boundary of \( X \) in the unique minimal affine subspace of \( \R^n \) which contains \( X \).

We adopt the convention that the Hausdorff measure of \( X \) in dimension \( d \) is normalized by \( 1/2^d \). That is, \[ \H^d(X)=\lim_{r\downto 0}\inf \left\{ \sum \left(\frac{\textrm{diam}(U_i)}{2}\right)^d :\, \medcup U_i\supset X,\, \textrm{diam}(U_i)< r \right\}. \] We say the set \( X \) is \emph{\textbf{\( d \)-rectifiable}} if there is a countable collection of Lipschitz maps \( f_i:\R^d\to \R^n \) such that \( X\setminus \medcup_i f_i(\R^d) \) is an \( \H^d \) null set. To simplify notation, we will also say \( X \) is \( (-1) \)-rectifiable if \( X \) is empty.

We will make use of the following additional terminology: If \( d \) is an integer such that \( 0\leq d\leq n \) and \( B \) is a closed ball in \( \R^n \), a \emph{\textbf{\( d \)-slice of \( B \)}} is a subset \( W \) of \( \R^n \) such that \( W\cap \textrm{int} (B) \) is non-empty, and such that \[ W= \left\{(x_1,\dots,x_n)\in \R^n: l_i\leq x_i\leq r_i \textrm{ for all } i\in I \textrm{, and } x_j=c_j \textrm{ for all } j\in \{1,\dots,n\}\setminus I \right\} \] for some \( I\subset \{1,\dots,n\} \) satisfying \( |I|=d \), where \[ -\i\leq l_i<r_i\leq\i \textrm{ and } c_j\in \R \textrm{ for all } i\in I \textrm{ and } j\in \{1,\dots,n\}\setminus I. \] Thus, \( W \) is a \( d \)-dimensional rectangle which may have infinite length in some or all of its directions. If \( i\in I \) and \( e_i \) is the \( i \)-th coordinate vector of \( \R^n \), we say the \emph{\textbf{width of \( W \) in the direction of \( e_i \)}} is \( r_i-l_i \). If \( 0\leq k\leq d \), we say the \emph{\textbf{\( k \)-width of \( W \)}} is \[ \inf_{\tilde{I}\subset I, |\tilde{I}|=k} \sup_{i\in \tilde{I}} (r_i-l_i), \] unless \( k=0 \), in which case we say the \( 0 \)-width of \( W \) is zero. Note that since \( W \) intersects the interior of \( B \), the set \( \fr(W\cap B) \) is, per our notation above, the boundary of \( W\cap B \) in the unique \( d \)-dimensional affine subspace of \( \R^n \) which contains \( W \).

If \( \e>0 \), let \( W(\e) \) denote the \( n \)-slice of \( B \) given by \( W+Q_\e \), where \( Q_\e \) is the coordinate \( (n-d) \)-cube orthogonal to \( W \) with side-length \( 2\e \) and centered at zero. If \( Z\subset B \) is closed, Let \( \fr(Z,W) \) denote the subset of \( Z\cap \fr(W\cap B) \) in the closure of \( Z\cap \textrm{int}\left(W(\e)\cap B\right) \) (noting that \( \fr(Z,W) \) is independent of the choice of \( \e>0 \).)

The notation \( \fr(Z,W) \) will first be used in Lemma \ref{lem:2}. Given a closed set \( Z\subset B \) and a \( d \)-slice \( W \) of \( B \), we will want to replace \( Z \) with a set \( \tilde{Z} \) such that \( \tilde{Z}\cap W \) is contained in \( C(Z\cap \fr(W\cap B),p) \) for some \( p\in W\cap B \). We will produce the set \( \tilde{Z} \) by deforming \( Z \) via a sequence of diffeomorphisms. We will perform this process inductively for a sequence of \( d \)-slices of \( B \). However, since \( Z\cap \fr(W\cap B) \) might have been altered by a previous deformation occurring on a slice \( W' \) of \( B \) adjacent to \( W \), we will want some finer control over the resulting set \( \tilde{Z}\cap W \). So, we will show that \( \tilde{Z}\cap W \) is actually contained in \( C(\fr(Z,W),p)\subset C(Z\cap \fr(W\cap B),p) \), the set \( \fr(Z,W) \) being unchanged by deformations occurring on adjacent slices.

\section{Lemmas and Constructions}
The following results will be used in the proof of Theorem \ref{thm:main}:

\begin{lem}
	\label{lem:1}
	Suppose \( (Z_i)_{i\in \N} \) is a sequence of compact subsets of \( \R^n \) and that \( (Z_i)_{i\in \N} \) converges in the Hausdorff metric to \( Z \). If \( \{f_j: \R^n\to \R^n \} \) is a sequence of continuous maps and \( (f_j(Z))_{j\in \N} \) converges in the Hausdorff metric to \( W \), then \( W \) is a limit point of \( (f_j(Z_i))_{(i,j)\in \N\times \N} \).
\end{lem}

\begin{proof}
	This follows immediately from the observation that for each \( j\in \N \), the sequence \( (f_j(Z_i))_{i\in \N} \) converges to \( f_j(Z) \) in the Hausdorff metric.
\end{proof}

\begin{lem}
	\label{lem:4}
	If \( X\subset \R^n \) is a Borel set and \( X_h \) denotes the points of \( X \) at distance \( h \) from a fixed affine hyperplane of \( \R^n \), then \[ \int_0^\i \H^{m-1} (X_h) \,d\H^1(h) \leq \H^m (X). \]
\end{lem}

\begin{proof}
	This is a special case of the Eilenberg inequality \cite{eilenberg}.
\end{proof}

\begin{lem}
	\label{lem:2}
	If \( Q \) is a \( d \)-slice of a closed ball \( B\subset \R^n \), where \( 0\leq d \leq n \), a point \( p \in Q\cap B \) is given, and \( A\subset B \) is closed, then for \( \e>0 \) there exists a sequence of diffeomorphisms \( (\phi_i)_{i\in \N} \) of \( \R^n \) fixing \( A\cup (\R^n\setminus B) \) such that if \( Z\subset B \) is closed and \( \fr(Z,Q)\subset A \), then any subsequential Hausdorff limit \( \tilde{Z} \) of \( (\phi_i(Z))_{i\in \N} \) satisfies the following properties:
	\begin{enumerate}
		\item\label{lem:2:item:1} The set \( \tilde{Z}\cap Q \) is contained in \( \{p\}\cup C(A\cap Q,p) \cup (Z\cap \fr(Q\cap B)); \)
		\item\label{lem:2:item:2} The set \( \tilde{Z} \) is equal to \( Z \) on the complement of \( \{p\}\cup\left(Q\cap\textrm{int}(B)\right)\cup \textrm{int}\left(Q(\e)\cap B\right); \) and,
		\item\label{lem:2:item:3} The set \( \fr(\tilde{Z},Q(\e)) \) is contained in \( \{p\}\cup C(A\cap Q,p)\cup \fr(Z,Q(\e)). \)
	\end{enumerate}
\end{lem}

\begin{proof}
	We proceed via a generalization of the approximate cone argument given in \cite{delellisandmaggi}. Let \( E \) denote the \( d \)-dimensional linear subspace of \( \R^n \) such that \( p+E \) contains \( Q \) and let \( G \) be the \( (n-d) \)-dimensional linear subspace of \( \R^n \) orthogonal to \( E \). Let \( 0<\d<\e \) be small enough so that if \( x\in G \) and \( |x|\leq\d \), then \( x+Q \) intersects \( \textrm{int}(B) \) non-trivially. Thus there exists a closed \( (n-d) \)-disk \( D \) smoothly embedded in \( B \), centered at \( p \), and transverse to \( p+E \) at \( p \) such that:
	\begin{enumerate}
		\item[\eqref{transverse}] \( D\setminus \{p\} \subset \textrm{int}(((N(0,\d)\cap G)+Q)\cap B); \) \eqnum\label{transverse}
		\item[\eqref{transverse2}] the orthogonal projection \( \pi \) of \( D \) onto \( p+G \) is a diffeomorphism onto its image; and\eqnum\label{transverse2}
		\item[\eqref{transverse3}] \( \pi(D) \) contains \( p+(N(0,3\d/4)\cap G) \).\eqnum\label{transverse3}
	\end{enumerate}
	
	The maps \( \phi_i \) can be realized as the time \( i \) flows of a fixed compactly supported \( C^\i \) vector field\footnote{Such a vector field \( V \) is complete, see e.g. \cite{sharpe} Ch. 2 Prop. 1.6, so \( V \) generates a \( 1 \)-parameter group of diffeomorphisms \( \{(\phi_t: \R^n\to \R^n) : t\in \R \} \). We restrict \( t \) to \( \N \) to generate our sequence of diffeomorphisms \( \{\phi_i  \} \).} \( V \) on \( \R^n \), which we shall now define.  
	
	Let \( \hat{V} \) be the \( C^\i \) vector field on \( D+E \) given by \( \hat{V}(y+e)=-e \) for all \( y\in D \) and \( e\in E \), noting that \( \hat{V} \) is well-defined by \eqref{transverse2}. Extend \( \hat{V} \) to a (discontinuous) vector field on \( \R^n \) by setting \( \hat{V}(x)=0 \) for all \( x\in \R^n\setminus (D+E) \).
	
	Let \( F \) denote the complement in \( \R^n \) of the interior of the set \[ ((N(0,\d/2)\cap G)+Q)\cap B \] and let \( \theta: \R^n\to \R \) be a non-negative \( C^\i \) function which vanishes precisely on the closed set \( D\cup A \cup F \). Such a function exists by a standard argument: The complement of \( D\cup A \cup F \) is a countable union of open balls \( B_k \) of radius \( r_k \) and center \( x_k \). Let \[ \theta_k(x)=\begin{cases}
	\mathrm{exp}\left(\frac{-1}{r_k^2-\|x-x_k\|^2}\right) & \mbox{if}\ x\in B_k;\\
	0 & \mbox{if} \ x\in \R^n\setminus B_k. 
	\end{cases} \]
	Let \( \theta(x)=\sum_{k\in \N} a_k \cdot \theta_k(x) \), where 
	\[ a_k=\begin{cases}
	\frac{1}{2^k}\left(\sup_{|\a|\leq k,\,|x|\leq k}|\partial^{\a}\theta_k(x)|\right)^{-1} & \mbox{if}\ \sup_{|\a|\leq k,\, |x|\leq k}|\partial^{\a}\theta_k(x)|\neq 0;\\
	\frac{1}{2^k} & \mbox{otherwise.}
	\end{cases}
	\]
	It is straightforward to check that the series defining \( \theta(x) \) converges in the Fr\'{e}chet topology on \( C^\i(\R^n) \). Let \( V=\theta\cdot \hat{V} \). The vector field \( V \) is compactly supported and \( C^\i \) by \eqref{transverse3} and by the definition of \( F \), and vanishes precisely on the set \( D\cup A\cup F \).

	We first prove Property \ref{lem:2:item:1}. If \( q\in \tilde{Z}\cap Q \), then there exists a sequence of points \( q_i\in Z \) such that \( \phi_i(q_i)\to q \). If there exist arbitrarily large \( i\in \N \) such that \( \phi_i(q_i)= q_i \), then \( q \) must be in the zero set of \( V \) and in \( Z\cap Q \). We then have \[ q\in \{p\}\cup (A\cap Q) \cup (Z\cap \fr(Q\cap B)) \] and we are done. Otherwise, by taking a subsequence if necessary, we may assume \( \phi_i(q_i)\neq q_i \) (thus \( V(q_i)\neq 0 \)) for all \( i\in \N \) and that \( q_i \) converges to some \( q'\in Z \). In particular, by the definition of \( V \), 
	\begin{equation}
		\label{ZQ}
		\mathrm{if}\ q'\in \fr(Q\cap B)\mathrm{, then}\ q'\in \fr(Z,Q).
	\end{equation}
	
	Suppose \( V(q)=0 \). Then either \( q\in \{p\}\cup (A\cap Q) \) and we are done, or \( q\in \fr(Q\cap B)\setminus \{p\} \). In this case, we must have \( q'\in \fr(Q\cap B)\setminus \{p\} \), for if not, then the line segments \( l(q_i,p_i) \) do not contain \( \phi_i(q_i) \) for large enough \( i \), where \( p_i \) denotes the point in \( D \) with \( p_i-q_i\in E \). This is a contradiction by the definition of \( V \). Thus, if \( q=q' \), then we are done by \eqref{ZQ}, otherwise the line segment \( l(q',p) \) contains \( q \), since \( l(q_i,p_i) \) contains \( \phi_i(q_i) \), and again we are done.
	
	Now suppose \( V(q)\neq 0 \). Then there is a small closed neighborhood \( N \) of \( q \) on which the magnitude of \( V \) has a positive lower bound, and since trajectories of \( V \) are straight lines, there is a maximum \( t<\i \) such that the image of \( N \) under the time \( t \) flow of \( V \) intersects non-trivially with \( N \). Thus, \( q'\notin N \) and in particular \( q'\neq q \). If \( V(q')=0 \) (hence \( q'\in A\cap Q \),) or if \( A\cap l(q',q)\neq \emptyset \), then we are done, otherwise there is a small neighborhood of \( l(q', q) \) on which the magnitude of \( V \) has a positive lower bound and this yields a contradiction, and so Property \ref{lem:2:item:1} is established.
	
	We now prove Property \ref{lem:2:item:2}. It is only possible for \( \tilde{Z} \) and \( Z \) to differ on the support of \( V \), so by \ref{lem:2:item:1}, we need only compare \( Z \) and \( \tilde{Z} \) at those points \( q\in \fr((w+Q)\cap B) \) for \( w\in G \) with \( 0<\|w\|<\d/2 \). Since \( V(q)=0 \), it follows that if \( q\in Z \) then \( q\in \tilde{Z} \). On the other hand, suppose \( q\in \tilde{Z} \). By \eqref{transverse}, the vector field \( V \) points away from \( \fr((w+Q)\cap B) \), so if \( q_i\in Z \) and \( \phi_i(q_i)\to q \), then \( q_i\to q \), so \( q\in Z \) as well. 
	
	We now prove property \ref{lem:2:item:3}. Suppose \( x\in \fr(\tilde{Z},Q(\e))\cap (w+Q) \) for some \( w\in G \setminus \{ 0\} \). Then the same argument as used to prove Property \ref{lem:2:item:2} proves that \( x\in \fr(Z,Q(\e)) \). Now suppose \( x\in \fr(\tilde{Z},Q(\e))\cap Q \). Our proof of Property \ref{lem:2:item:1} actually shows that \( \tilde{Z}\cap (w+Q) \) is contained in \[ \{\pi^{-1}(w+p)\}\cup C((A\cap(w+Q))\cup \fr(Z,w+Q),\pi^{-1}(w+p)) \cup (Z\cap \fr((w+Q)\cap B)) \] for \( w\in G \setminus \{ 0\} \) small enough. Therefore, there exist \( x_i\to x \) and \( w_i\to 0 \) with \[ x_i\in \{\pi^{-1}(w_i+p)\}\cup C((A\cap (w_i+Q))\cup \fr(Z,w_i+Q),\pi^{-1}(w_i+p)), \] and so there exists \( y_i\in \{\pi^{-1}(w_i+p)\}\cup(A\cap(w_i+Q))\cup \fr(Z,w_i+Q) \) so that \( x_i\in l(y_i,\pi^{-1}(w_i+p)\}) \). Taking a subsequence if necessary, we have \( y_i\to y\in \{p\}\cup (A\cap Q) \), and so \( x\in l(y,p) \) and we are done.
\end{proof}

\begin{lem}
	\label{lem:6}
	If \( r>0 \), \( p\in \R^n \) and \( X\subset N(p,r) \), then \[ \H^m (C(X,p))\leq 2^{2m-1} \cdot r\cdot \H^{m-1} (X). \]
\end{lem}

\begin{proof}
	This is a special case of \cite{reifenberg} Lemma 6, Chapter I.
\end{proof}

\section{Proof of Theorem \ref{thm:main}}

	To prove Theorem \ref{thm:main}, we shall take the set \( Y \) and successively apply Lemma \ref{lem:2} to it in various ways, using \( Z=Y \) as the input in the first application of Lemma \ref{lem:2}, the set \( \tilde{Z} \) as the input for the second, \( \tilde{\tilde{Z}} \) as the third, etc. The desired set \( \tilde{Y} \) will be the final set in this finite sequence. The choice of data used to produce the sequence of diffeomorphisms at each stage will be independent of \( Y \), and this will be sufficient by Lemma \ref{lem:1} to produce a single sequence of diffeomorphisms \( \phi_i \) as desired.
	
	The procedure will mirror the construction in the proof of \cite{reifenberg} Lemma 8, Chapter 1, in that the set \( X \) in \cite{reifenberg} Lemma 8, Chapter I is constructed via a sequence of cones over sets inductively constructed on orthogonal affine subspaces of decreasing codimension. At each stage, we will replace the cone used by Reifenberg with the construction in Lemma \ref{lem:2}. The proof of Theorem \ref{thm:main} will follow the proof of \cite{reifenberg} Lemma 8, Chapter 1 for the most part, with the exception that for the sake of brevity we combine the \( m=2 \) and \( m>2 \) cases into a single argument. See also footnote \ref{Rproof}. 
	
	However, because of the global nature of the constructions involved, our proof will necessarily be significantly more complex. Much of the difficulty will lie in ensuring the set \( \tilde{Y} \) is in the convex hull of \( A \), because doing so requires us to pick cone points on the boundary of the regions in which we are coning. Producing a sequence of diffeomorphisms which is independent of \( Y \) is also a source of complexity. However, neither of these properties are necessary for the application of Theorem \ref{thm:main} in \cite{elliptic}.
	
\begin{proof}[proof of Theorem \ref{thm:main}]
	For integers \( 2\leq m \leq N \leq n \) and \( 0\leq k\leq N \), let \( P(m,k,N,n) \) be the following proposition: \emph{There exists a constant \( 0<K^m_k<\i \) depending only on \( m \) and \( k \) such that if we are given a closed subset \( A \) of a ball \( B\subset \R^n \), an \( N \)-slice \( W \) of \( B \), an upper bound \( 0<L\leq \i \) for both the \( k \)-width of \( W \) and for the quantity \( 4\cdot\H^{m-1}(A\cap W)^{1/(m-1)} \), a closed \( (m-2) \)-rectifiable set \( S\subset A \) of finite \( \H^{m-2} \) measure, and some \( \zeta^m_k>0 \), then there exist: a closed set \( \tilde{A}\subset W\cap B \) such that
	\begin{enumerate}
		\item\label{prop:item:0} \( \tilde{A}\cap \fr B\subset A\cap \fr B \),
		\item\label{prop:item:1} \( \H^m(\tilde{A}) \leq K^m_k\cdot L\cdot \H^{m-1}(A\cap W), \) and
		\item\label{prop:item:2} if \( A\cap W \) is \( (m-1) \)-rectifiable then \( \tilde{A} \) is \( m \)-rectifiable;
	\end{enumerate}
	a closed \( (m-1) \)-rectifiable set \( \tilde{S}\subset \tilde{A} \) of finite \( \H^{m-1} \) measure such that
	\begin{enumerate}[resume]
		\item\label{prop:item:3} \( \tilde{A}\setminus \tilde{S} \subset \co\left((A\cap W)\setminus S\right)\cap N\left((A\cap W)\setminus S, K^m_k\cdot L \right); \)
	\end{enumerate} 
	and a sequence of diffeomorphisms \( (\phi_i)_{i\in \N} \) of \( \R^n \) fixing \( A\cup(\R^n\setminus B) \), such that if \( Y\subset B \) is closed and \( \fr(Y,W)\subset A \), then there exists a limit point \( \tilde{Y} \) of \( (\phi_i(Y))_{i\in \N} \) such that
	\begin{enumerate}[resume]
		\item\label{prop:item:4} the set \( \tilde{Y}\cap W \) is contained in \( \tilde{A}\cup (Y\cap \fr(W\cap B)), \)
		\item\label{prop:item:5} the set \( \tilde{Y} \) is equal to \( Y \) on the complement of \[ (A\cap W\cap \fr B)\cup\left(W\cap\textrm{int}(B)\right)\cup \textrm{int}\left(W(\zeta^m_k)\cap B\right), \] and
		\item\label{prop:item:6} the set \( \fr(\tilde{Y},W(\zeta^m_k)) \) is contained in \( \left(\tilde{A}\cap \fr(W\cap B)\right)\cup \fr(Y,W(\zeta^m_k)). \)
	\end{enumerate}}
	
	For integers \( 1\leq N\leq n \), let \( P(1,0,N,n) \) be the following proposition: \emph{If we are given a closed subset \( A \) of a ball \( B\subset \R^n \), an \( N \)-slice \( W \) of \( B \) such that \( A\cap W \) is empty, and some \( \zeta^1_0>0 \), then there exist a finite subset of \( W\cap B \) which we will denote by both \( \tilde{A} \) and \( \tilde{S} \), and a sequence of diffeomorphisms \( (\phi_i)_{i\in \N} \) of \( \R^n \) fixing \( A\cup(\R^n\setminus B) \), such that if \( Y\subset B \) is closed and \( \fr(Y,W) \) is empty, then there exists a limit point \( \tilde{Y} \) of \( (\phi_i(Y))_{i\in \N} \) such that
	\begin{enumerate}
		\item[\( (e') \)]\label{statement:item:1} the set \( \tilde{Y}\cap W  \) is contained in \( \tilde{A}\cup (Y\cap \fr(W\cap B)), \)
		\item[\( (f') \)]\label{statement:item:2} the set \( \tilde{Y} \) is equal to \( Y \) on the complement of \[ \left(W\cap\textrm{int}(B)\right)\cup \textrm{int}\left(W(\zeta^1_0)\cap B\right), \] and
		\item[\( (g') \)]\label{statement:item:3} the set \( \fr(\tilde{Y},W(\zeta^1_0)) \) is contained in \( \left(\tilde{A}\cap \fr(W\cap B)\right)\cup \fr(Y,W(\zeta^1_0)). \)
	\end{enumerate}}

	The proposition \( P(1,0,N,n) \) is implied by Lemma \ref{lem:2} applied to \( Q=W \), \( d=N \), any point \( p\in W\cap\textrm{int}(B) \) and \( \e=\zeta^1_0 \), letting \( \tilde{A}=\tilde{S}=\{p\} \).
	
	We now prove \( P(m,N,N,n) \) for \( 2\leq m \leq N \leq n \). In this case, \( W \) is an \( N \)-rectangle of side-length \( \leq L \). Let \( p \in (A\cap W)\setminus S \) if \( (A\cap W)\setminus S \) is non-empty, otherwise let \( p \) be an arbitrary point of the non-empty set \( W\cap\textrm{int}(B) \). Let \[ \tilde{A}= \{p\}\cup C(A\cap W,p) \] and let \( K^m_N=2^{2m-1}\sqrt{N} \). Property \ref{prop:item:0} is satisfied due to our choice of the point \( p \). Since \( W \) is contained in an \( N \)-cube of side length \( \leq L \), Property \ref{prop:item:1} is satisfied by Lemma \ref{lem:6}. Since the cone \( C(X,p) \) of an \( (m-1) \)-rectifiable set \( X \) is \( m \)-rectifiable, the set \( \tilde{A} \) is \( m \)-rectifiable whenever \( A\cap W \) is \( (m-1) \)-rectifiable, so Property \ref{prop:item:2} is satisfied.
	
	Let \( \tilde{S} = C(S\cap W,p)\subset \tilde{A}, \) which is a closed \( (m-1) \)-rectifiable set of finite \( \H^{m-1} \) measure. The set \( \tilde{A}\setminus \tilde{S} \) is contained by construction in \[ \co((A\cap W)\setminus S)\cap N((A\cap W)\setminus S, \sqrt{N}\cdot L) \] and so Property \ref{prop:item:3} is satisfied.
	
	Let \( (\phi_i)_{i\in \N} \) denote the sequence of diffeomorphisms produced by Lemma \ref{lem:2} applied to \( Q=W \), \( d=N \), the point \( p \) above, and using \( \e=\zeta^m_N \). Properties \ref{prop:item:4}-\ref{prop:item:6} follow from Lemma \ref{lem:2} and our choice of point \( p \) above. So, we have established \( P(m,N,N,n) \).
	
	Now for \( 2\leq m \leq N \leq n \) and \( 1\leq k\leq N \), we will assume \( P(m,k,N,n) \) and \( P(m-1,0,N-1,n) \) and deduce \( P(m,k-1,N,n) \). Since we know \( P(m,N,N,n) \), by downward induction on \( k \), this will prove the implication \( P(m-1,0,N-1,n) \Rightarrow P(m,0,N,n) \). Since we know \( P(1,0,n-(m-1),n) \), this will prove \( P(m,0,n,n) \) by induction on \( m \). The proposition \( P(m,0,n,n) \) implies Theorem \ref{thm:main}.

	So, suppose \( (A,\, B,\, W,\, L,\, S,\, \zeta^m_{k-1}) \) satisfies the assumptions of \( P(m,k-1,N,n) \). If the \( k \)-width of \( W \) is bounded above by \( L \), then \( (A,\, B,\, W,\, L,\, S,\, \zeta^m_{k-1}) \) satisfies the assumptions of \( P(m,k,N,n) \) and the resulting sets \( \tilde{A} \), \( \tilde{S} \) and sequence of diffeomorphisms \( (\phi_i)_{i\in \N} \) satisfy the conclusions of \( P(m,k-1,N,n) \), and so we are done. So, let us assume the \( k \)-width of \( W \) is strictly greater than \( L \), and in particular that \( L<\i \). Let \( G \) denote the unique \( N \)-dimensional affine subspace of \( \R^n \) which contains \( W \) and let \( x \) be a unit coordinate vector of \( \R^n \) which is parallel to \( G \), such that the width of \( W \) in the direction of \( x \) is strictly greater than \( L \). 
	
	Let \( \G \) be an affine hyperplane of \( G \) orthogonal to \( x \) and consider the map
	\begin{align*}
		\pi_\G:G&\to \R^n\\
		y&\mapsto y-L\,\lfloor \mathrm{dist}(y,\G)/L\rfloor \cdot x,
	\end{align*}
	where \( \mathrm{dist}(y,\G) \) is understood to be the signed distance from \( y \) to \( \G \), with positive (resp. negative) sign if \( y \) differs from a vector of \( \G \) by a positive (resp. negative) multiple of \( x \). There exists, by Lemma \ref{lem:4} applied to \( \pi_\G(A\cap W) \) and \( \pi_\G(S\cap W) \), a bi-infinite sequence \( \Sigma_\a \) of mutually disjoint, parallel affine hyperplanes of \( G \), perpendicular to \( x \), each of which is distance \( L \) from the preceding and next hyperplane in the sequence, whose union \( \Sigma \) satisfies
	\begin{equation}
		\label{inequality:-1}
		\H^{m-2}(A\cap W\cap \Sigma)\leq 2\cdot L^{-1}\cdot\H^{m-1}(A\cap W) \leq 2^{3-2m}\cdot L^{m-2}<\i
	\end{equation}
	and\footnote{In the case that \( m=2 \) we encounter the quantity \( \H^{-1}(S\cap W\cap \Sigma) \), which is zero if and only if \( S\cap W\cap\Sigma \) is empty, and otherwise is infinite.}
	\begin{equation}
		\label{nothing}
		\H^{m-3}(S\cap W\cap \Sigma)<\i.
	\end{equation}
	Moreover by \cite{federer} 3.2.22(2) applied to the map \[ y\mapsto \mathrm{dist}(y,\G) \] on the sets \( \pi_\G(S\cap W) \) and \( \pi_\G(A\cap W) \), noting that a set \( X\subset G \) is \( k \)-rectifiable if and only if \( \pi_\G(X) \) is \( k \)-rectifiable, we can choose \( \Sigma_\a \) so that \( S\cap W\cap \Sigma \) is \( (m-3) \)-rectifiable, and so that in the case \( A\cap W \) is \( (m-1) \)-rectifiable, the set \( A\cap W\cap \Sigma \) is \( (m-2) \)-rectifiable. 
	
	Relabeling, let \( \{ \Sigma_j : j\in J:=\{0,\dots,M+1\} \} \) denote the set of consecutive hyperplanes of \( \Sigma_\a \) such that for each \( j\in J\setminus \{0,M+1\} \), the set \( (W\cap B)\setminus\Sigma_j \) is disconnected, and such that \( W\cap B \) is contained in the closed region of \( G \) bounded by \( \Sigma_0 \) and \( \Sigma_{M+1} \).
	
	Note that for each \( j\in J\setminus \{0,M+1\} \), the set \( W\cap \Sigma_j \) is a \( (N-1) \)-slice of \( B \), and that if \( Y\subset B \) is closed and \( \fr(Y,W)\subset A \), then
	\begin{equation}
		\label{trick}
		\fr(Y,W\cap \Sigma_j)=\fr(Y,W)\cap \Sigma_j\subset A.
	\end{equation}
	Also note that if \( m=2 \) then \eqref{inequality:-1} implies \( A\cap W\cap \Sigma \) is empty, and that if \( m>2 \) then \[ 4\cdot \H^{m-2}(A\cap W\cap \Sigma)^{1/(m-2)} \leq L. \] So, we may apply \( P(m-1,0,N-1,n) \) using, for each \( j\in J\setminus \{0,M+1\} \), the input data \[ (A,\,B,\, W\cap \Sigma_j,\, L ,\, S,\, \zeta^{m-1}_0=\textrm{min}(L/4,\zeta^m_{k-1}/2)), \] to produce: a closed set \( X_j\subset W\cap \Sigma_j\cap B \) for each \( j\in J\setminus \{0,M+1\} \), such that
	\begin{enumerate}
		\item[\eqref{sub:item:0}] The set \( X_j\cap \fr B \) is contained in \( A\cap \fr B \), \eqnum\label{sub:item:0}
		\item[\eqref{sub:item:1}] \( \H^{m-1}(X_j) \leq K^{m-1}_0\cdot L\cdot \H^{m-2}(A\cap W\cap \Sigma_j), \)  and \eqnum\label{sub:item:1}
		\item[\eqref{sub:item:2}] If \( A\cap W \) is \( (m-1) \)-rectifiable, then \( X_j \) is also \( (m-1) \)-rectifiable; \eqnum\label{sub:item:2}
	\end{enumerate}
	a closed \( (m-2) \)-rectifiable set \( T_j\subset X_j \) of finite \( \H^{m-2} \) measure for each \( j\in J\setminus \{0,M+1\} \), such that
	\begin{enumerate}
		\item[\eqref{sub:item:3}] The set \( X_j\setminus T_j \) is contained in \[ \co\left((A\cap W\cap \Sigma_j)\setminus S\right)\cap N\left((A\cap W\cap \Sigma_j)\setminus S, K^{m-1}_0\cdot L\right); \] \eqnum\label{sub:item:3}
	\end{enumerate}
	and a sequence of diffeomorphisms \( (\psi_i^j)_{i\in \N} \) of \( \R^n \) for each \( j\in J\setminus \{0,M+1\} \), such that if \( (\psi_i)_{i\in \N} \) is an enumeration of \[ (\psi_{i_M}^M\circ\cdots\circ\psi_{i_1}^1)_{(i_1,\dots, i_M) \in \N\times\cdots\times \N} \] in the case that \( M\geq 1 \), otherwise if \( \psi_i=\textrm{Id} \) for all \( i\in \N \), then \( \psi_i \) fixes \( A\cup(\R^n\setminus B) \) for all \( i\in \N \) and if \( Y\subset B \) is closed and satisfies \( \fr(Y,W)\subset A \), then there exists by \eqref{trick}, Lemma \ref{lem:1}, our choice of \( \zeta^{m-1}_0<L/2 \), and \( P(m-1,0,N-1,n)\)\ref{prop:item:5} (resp. \( (f') \) if \( m=2 \),) a limit point \( Y' \) of \( (\psi_i(Y))_{i\in \N} \) such that
	\begin{enumerate}
		\item[\eqref{sub:item:4}] For each \( j\in J\setminus \{0,M+1\} \), \[ Y'\cap W\cap \Sigma_j \subset X_j\cup (Y\cap \fr(W\cap \Sigma_j \cap B)), \] \eqnum\label{sub:item:4}
		\item[\eqref{sub:item:5}] The set \( Y' \) is equal to \( Y \) on the complement of \[ (A\cap W\cap\Sigma\cap \fr B)\cup\left(W\cap\Sigma\cap\textrm{int}(B)\right)\cup \medcup_{j\in J\setminus \{0,M+1\}} \textrm{int}\left((W\cap \Sigma_j)(\zeta^{m-1}_0)\cap B\right), \textrm{ and} \] \eqnum\label{sub:item:5}
		\item[\eqref{sub:item:6}] For each \( j\in J\setminus \{0,M+1\} \), \[ \fr(Y',(W\cap \Sigma_j)(\zeta^{m-1}_0)) \subset \left(X_j\cap \fr(W\cap\Sigma_j\cap B)\right)\cup \fr(Y,(W\cap \Sigma_j)(\zeta^{m-1}_0)). \] \eqnum\label{sub:item:6}
	\end{enumerate}

	To simplify our notation later on, let \( X_0=T_0=X_{M+1}=T_{M+1}=\emptyset \). For \( j=\{0,\dots,M\} \), let \( W_j \) denote the closed region in \( W \) bounded by \( \Sigma_j \) and \( \Sigma_{j+1} \). Note that each such \( W_j \) is an \( N \)-slice of \( B \) whose \( k \)-width is bounded above by \( L \).
	
	Let \( \cal{A}=A\cup \cup_{j\in J} X_j \) and \( \cal{S}=S\cup\cup_{j\in J} T_j \). 
	
	A short calculation using \eqref{sub:item:4}, \eqref{sub:item:5} and \eqref{sub:item:6} shows that if \( Y\subset B \) is closed and \( \fr(Y,W)\subset A \), then
	\begin{equation}
		\label{trick3}
		\fr(Y',W_j) \subset \left(\fr(Y,W)\cap W_j \right) \cup X_j\cup X_{j+1}\subset (A\cap W_j)\cup X_j\cup X_{j+1} \subset \cal{A}.
	\end{equation}
	
	Also observe that by \eqref{sub:item:4} and \eqref{sub:item:5},
	\begin{equation}
		\label{trick4}
		Y'\cap \fr(W_j\cap B) \subset \left( Y\cap \fr(W\cap B) \right) \cup X_j \cup X_{j+1}.
	\end{equation}
	
	In a moment, we shall apply \( P(m,k,N,n) \) in sequence for each \( j=\{0,\dots,M\} \) using the input data \[ (\cal{A},\, B,\, W_j,\, \cal{L}_m,\, \cal{S},\, \zeta^m_k=\zeta^{m-1}_0/2), \] where\footnote{\label{Rproof}There is a small error at the corresponding point of Reifenberg's original proof of Theorem \ref{thm:lem8}. On page 17, his set \( X_i+X_{i+1}+A_i \), which corresponds to our set \( X_i\cup X_{i+1}\cup (A\cap W_i) \) does not necessarily satisfy the hypotheses of his proposition \( P(m_0,N,k_0) \), since the quantity \( \Lambda^{m-1}(X_i+X_{i+1}+A_i)^{1/(m-1)} \) may be too small. We fix this issue by introducing the quantity \( \cal{L}_m \) in our proof.} \( \cal{L}_m \) is an upper bound of \( L \) and \( 4\cdot\H^{m-1}(\cal{A}\cap W_j)^{1/(m-1)} \), determined as follows.
	
	By definition, \( \cal{A}\cap W_j=(A\cap W_j) \cup X_j \cup X_{j+1} \). We have by \eqref{inequality:-1} and \eqref{sub:item:1},
	\begin{align*}
		\H^{m-1}(\cal{A}\cap W_j)^{1/(m-1)} &\leq  \H^{m-1}(A\cap W_j)^{1/(m-1)} + \H^{m-1}(X_j)^{1/(m-1)} + \H^{m-1}(X_{j+1})^{1/(m-1)}\\
		&\leq \H^{m-1}(A\cap W)^{1/(m-1)} + 2\cdot \left(K_0^{m-1}\cdot L\cdot \H^{m-2}(A\cap W\cap \Sigma) \right)^{1/(m-1)}\\
		&\leq \H^{m-1}(A\cap W)^{1/(m-1)} + 2\cdot\left(K_0^{m-1}\cdot 2\cdot \H^{m-1}(A\cap W)\right)^{1/(m-1)}.
	\end{align*}
	Thus, by the definition of \( L \),
	\begin{align*}
		4\cdot\H^{m-1}(\cal{A}\cap W_j)^{1/(m-1)} \leq L+2\cdot\left(K^{m-1}_0\cdot 2\right)^{1/(m-1)} \cdot L.
	\end{align*}
	Thus, we may use
	\begin{equation}
		\label{calL2}
		\cal{L}_m:=\left(1+2\cdot(K^{m-1}_0\cdot 2\,)^{1/(m-1)} \right)\cdot L.
	\end{equation}
	Therefore, by \( P(m,k,N,n) \), using for each \( j=\{0,\dots,M\} \) the input data \[ (\cal{A},\, B,\, W_j,\, \cal{L}_m,\, \cal{S},\, \zeta^m_k=\zeta^{m-1}_0/2), \] there exist: a closed set \( \tilde{\cal{A}}_j\subset W_j\cap B \) for each \( j=\{0,\dots,M\}, \) such that
	\begin{enumerate}
		\item[\eqref{chop0}] The set \( \tilde{\cal{A}}_j\cap \fr B \) is contained in \( \cal{A}\cap \fr B \), \eqnum\label{chop0}
		\item[\eqref{chop1}] The inequality \( \H^m(\tilde{\cal{A}}_j) \leq K^m_k\cdot \cal{L}_m\cdot \H^{m-1}(\cal{A}\cap W_j) \) holds, and \eqnum\label{chop1}
		\item[\eqref{chop2}] If \( A\cap W \) is \( (m-1) \)-rectifiable, then \( \tilde{\cal{A}}_j \) is \( m \)-rectifiable by \eqref{sub:item:2}; \eqnum\label{chop2}
	\end{enumerate}
	a closed \( (m-1) \)-rectifiable set \( \tilde{\cal{S}}_j\subset \tilde{\cal{A}}_j \) of finite \( \H^{m-1} \) measure for each \( j\in\{0,\dots,M\} \), such that
	\begin{enumerate}
		\item[\eqref{chop3}] The set \( \tilde{\cal{A}}_j\setminus \tilde{\cal{S}}_j \) is contained in \[ \co\left((\cal{A}\cap W_j)\setminus \cal{S}\right)\cap N\left((\cal{A}\cap W_j)\setminus \cal{S}, K^m_k\cdot \cal{L}_m\right); \textrm{ and} \] \eqnum\label{chop3}
	\end{enumerate}
	and a sequence of diffeomorphisms \( (\phi_i^j)_{i\in \N} \) of \( \R^n \) for each \( j\in \{0,\dots,M\} \), such that if \( (\phi_i)_{i\in \N} \) is an enumeration of \[ \left(\phi_{i_M}^M\circ\dots\circ\phi_{i_0}^0\circ \psi_{\iota}\right)_{(\iota,i_0,\dots,i_M)\in \N\times(\N\times\dots\times \N)}, \] then \( \phi_i \) fixes \( A\cup (\R^n\setminus B) \) for all \( i\in \N \) (since \( A\subset \cal{A} \),) and by Lemma \ref{lem:1}, \eqref{trick3} and \( P(m,k,N,n) \)\ref{prop:item:5}, if \( Y\subset B \) is closed and \( \fr(Y,W)\subset A \), then there exists a limit point \( \tilde{Y} \) of \( (\phi_i(Y))_{i\in \N} \) such that 
	\begin{enumerate}
		\item[\eqref{chop4}] By \eqref{trick4}, the set \( \tilde{Y}\cap W_j \) is contained in \[ \tilde{\cal{A}}_j \cup \tilde{\cal{A}}_{j-1} \cup \left( Y\cap \fr(W\cap B) \right) \cup X_j \cup X_{j+1} \] for each \( j=\{0,\dots,M\} \), where \( \tilde{\cal{A}}_{-1} \) denotes the empty set, \eqnum\label{chop4}
		\item[\eqref{chop5}] The set \( \tilde{Y} \) is equal to \( Y \) on the complement of \[ (A \cap W\cap \fr B) \cup (W\cap \textrm{int}(B))\cup \textrm{int}(W(\zeta^m_{k-1})\cap B) \] by \( P(m,k,N,n) \)\ref{prop:item:5}, \eqref{sub:item:0}, \eqref{sub:item:5} and since \( \zeta^m_k<\zeta^{m-1}_0<\zeta^m_{k-1} \), and \eqnum\label{chop5}
		\item[\eqref{chop6}] The set \( \fr(\tilde{Y}, W_j(\zeta^m_k)) \) is contained in \[ \left(\tilde{A}\cap \fr(W\cap B) \right)\cup \fr(Y, W(\zeta^m_{k-1})), \] where \[ \tilde{A}:=(\cal{A}\cap W)\cup\medcup_{j\in \{0,\dots,M\}} \tilde{\cal{A}}_j, \] by \( P(m,k,N,n) \)\ref{prop:item:5}, \( P(m,k,N,n) \)\ref{prop:item:6}, \eqref{sub:item:5}, \eqref{sub:item:6}, and since \( \zeta^m_k<\zeta^{m-1}_0<\zeta^m_{k-1} \). \eqnum\label{chop6}
	\end{enumerate}
	
	Therefore, Properties \( P(m,k-1,N,n) \)\ref{prop:item:5} and \( P(m,k-1,N,n) \)\ref{prop:item:6} are satisfied. Property \( P(m,k-1,N,n) \)\ref{prop:item:0} follows from \eqref{chop0} and \eqref{sub:item:0}. Property \( P(m,k-1,N,n) \)\ref{prop:item:2} follows from \eqref{inequality:-1}, \eqref{sub:item:1} and \eqref{chop2}. Property \( P(m,k-1,N,n) \)\ref{prop:item:4} follows from \eqref{chop4}. We now show that Property \( P(m,k-1,N,n) \)\ref{prop:item:1} is satisfied for \( K^m_{k-1}<\i \) large enough:
	
	It follows from \( P(m,k,N,n) \)\ref{prop:item:1}, \eqref{inequality:-1}, and \eqref{sub:item:1}, that
	\begin{align*}
		\H^m(\tilde{A})&\leq K^m_k\cdot \cal{L}_m\cdot \sum_{j=0}^M \H^{m-1}(\cal{A}\cap W_j)\\
		&\leq K^m_k\cdot \cal{L}_m\cdot \sum_{j=0}^M \left(\H^{m-1}(A\cap W_j) + \H^{m-1}(X_j) + \H^{m-1}(X_{j+1})\right)\\
		&\leq K^m_k\cdot \cal{L}_m\cdot 2\cdot \left(\H^{m-1}(A\cap W) + \sum_{j=0}^M \H^{m-1}(X_j)\right)\\
		&\leq K^m_k\cdot \cal{L}_m\cdot 2\cdot \left(\H^{m-1}(A\cap W) + \sum_{j=0}^M K_0^{m-1} \cdot L \cdot \H^{m-2}(A\cap W\cap \Sigma_j) \right)\\
		&\leq K^m_k\cdot \cal{L}_m\cdot 2\cdot \left(\H^{m-1}(A\cap W) + 2\cdot K_0^{m-1} \cdot \H^{m-1}(A\cap W) \right)\\
		&= K^m_k\cdot \cal{L}_m\cdot 2\cdot \left(1 + 2\cdot K_0^{m-1} \right)\cdot \H^{m-1}(A\cap W).
	\end{align*}
	
	So, by \eqref{calL2}, Property \( P(m,k-1,N,n) \)\ref{prop:item:1} is satisfied as long as we define \[ K^m_{k-1} \geq K^m_k\cdot\left(1+2\cdot(K^{m-1}_0\cdot 2\,)^{1/(m-1)} \right)\cdot 2 \cdot \left(1 + 2\cdot K_0^{m-1} \right). \]

	Let us finally attack Property \( P(m,k-1,N,n) \)\ref{prop:item:3}: Let \[ \tilde{S}:=(\cal{S}\cap W)\cup \cup_{j\in \{0,\dots M\}} \tilde{\cal{S}_j}, \] which is by construction a closed, \( (m-1) \)-rectifiable subset of \( \tilde{A} \) of finite \( \H^{m-1} \) measure.
	
	By \eqref{sub:item:3}, we know that \( (\cal{A}\cap W)\setminus \cal{S} \subset \co((A\cap W)\setminus S) \) and by \eqref{chop3} that \( \tilde{A}\setminus\tilde{S} \subset \co((\cal{A}\cap W)\setminus \cal{S}) \). Therefore, \[ \tilde{A}\setminus\tilde{S} \subset \co((A\cap W)\setminus S). \]

	Likewise, by \eqref{chop3}, we have \[ \tilde{A}\setminus \tilde{S} \subset N\left((\cal{A}\cap W)\setminus \cal{S}, K^m_k\cdot \cal{L}_m\right). \] By \eqref{sub:item:3}, we have \[ (\cal{A}\cap W)\setminus \cal{S}\subset N\left((A\cap W)\setminus S, K^{m-1}_0\cdot L  \right), \]
	
	
	Therefore, by \eqref{calL2}, Property \( P(m,k-1,N,n) \)\ref{prop:item:3} holds if \[ K^m_{k-1} \geq K^m_k\cdot \left(1+2\cdot\left(K^{m-1}_0\cdot 2\right)^{1/(m-1)}\right) + K_0^{m-1}. \] Choosing such an appropriate constant \( K^m_{k-1}<\i \), the proposition \( P(m,k-1,N,n) \) holds.	
	\end{proof}

\addcontentsline{toc}{section}{References} 
\bibliography{bibliography.bib}{}
\bibliographystyle{amsalpha}
\end{document}